\documentclass[12pt]{article}
 
\usepackage[margin=1in]{geometry} 
\usepackage{amsmath,amsthm,amssymb}
\usepackage[english]{babel}
\usepackage[utf8]{inputenc}
\usepackage{tikz-cd}
\usepackage{amsmath}
\usepackage{amsmath}
\usepackage{abraces}
\usepackage{caption}
\usepackage{subcaption}
\usepackage{enumerate,latexsym}

\def\R{\mathbb{R}}

\def\N{\mathbb{N}}
\def\Z{\mathbb{Z}}
\def\S{\Sigma}
\def\C{\mathbb{C}}
\def\Q{\mathbb{Q}}

\def\Pe{\mathbb{P}}

\def\Pe{\mathbb{P}}
\def\esf{\mathbb{S}}

\newcommand{\la}{\looparrowright}

\newcommand{\ben}{\begin{enumerate}}
	\newcommand{\bit}{\begin{itemize}}
		\newcommand{\een}{\end{enumerate}}
	\newcommand{\eit}{\end{itemize}}
\newcommand{\wh}{\widehat}

\newcommand{\Inj}{\mbox{\rm Inj}}

\newcommand{\wt}{\widetilde}

\newcommand{\ed}{\end{document}}

\def\t{{\theta}}

\def\g{{\gamma}}

\def\l{{\lambda}}

\def\be{{\beta}}
\def\ve{{\varepsilon}}

\newcommand\restr[2]{{
  \left.\kern-\nulldelimiterspace 
  #1 
  \vphantom{\big|} 
  \right|_{#2}
  }}
 
\newtheorem{theorem}{Theorem}
\newtheorem{lemma}[theorem]{Lemma}
\newtheorem{proposition}[theorem]{Proposition}

\newtheorem{remark}[theorem]{Remark}

\theoremstyle{definition}
\newtheorem{definition}[theorem]{Definition}

\begin{document}
 
\begin{title}
	{Generalized Henneberg stable minimal surfaces}
\end{title}

\begin{author}
{David Moya$^*$ \and Joaqu\'\i n P\' erez\thanks{This work is supported in 
part by the IMAG–Maria de Maeztu grant CEX2020-001105-M / AEI / 
10.13039/501100011033,
MICINN grant PID2020-117868GB-I00 and Junta de Andalucía grants P18-FR-4049 and A-FQM-139-UGR18.}}
\end{author}
\date{}

\maketitle
\begin{abstract}
We generalize the classical Henneberg minimal surface by giving an infinite family of complete, finitely
branched, non-orientable, stable minimal surfaces in 
$\R^3$. These surfaces can be grouped into subfamilies
depending on a positive integer (called the {\it 
complexity}), which essentially measures the number of
branch points. The classical Henneberg surface $H_1$ is
characterized as the unique example in the subfamily of 
the simplest complexity $m=1$, while for $m\geq 2$ 
multiparameter families are given. The isometry group 
of the most symmetric example $H_m$ with a given complexity $m\in \N$ is either isomorphic to the 
dihedral isometry group $D_{2m+2}$ (if $m$ is odd) or
to $D_{m+1}\times \Z_2$ (if $m$ is even). Furthermore, 
for $m$ even $H_m$ is the unique solution to the Björling
problem for a hypocycloid of $m+1$ cusps (if $m$ is even),
while for $m$ odd the conjugate minimal surface $H_m^*$
to $H_m$ is the unique solution to the Björling
problem for a hypocycloid of $2m+2$ cusps.
\end{abstract}

\section{Introduction}	
A celebrated result obtained independently by do Carmo \& Peng~\cite{cp1}, 
Fischer-Colbrie \& Schoen~\cite{fs1} and Pogorelov~\cite{po1} establishes
that if $M$ is a complete orientable stable minimal surface in $\R^3$, then
$M$ is a plane. Ros~\cite{ros9} proved that the same characterization holds
without assuming orientability. Nevertheless, a plethora of complete
stable minimal surfaces in $\R^3$ appear if we allow these stable minimal
surfaces to have branch points, with the simplest example being the 
classical Henneberg minimal surface~\cite{hen1}.

The class of complete, finitely connected and finitely branched minimal 
surfaces with finite total curvature (among which stable ones are a particular case) appears naturally in the following 
situation: Given $\ve_0>0$,  $I\in \N\cup \{0\}$ and $H_0,K_0\geq 0$,
let $\Lambda=\Lambda (I, H_0,\ve_0,K_0)$ be the set of immersions 
$F\colon M\la X$ where $X$ is a complete Riemannian 3-manifold with injectivity radius $\Inj(X)\geq \ve_0$ and absolute sectional curvature bounded from above by $K_0$, $M$ is a complete surface, $F$
has constant mean curvature $H\in [0,H_0]$ and Morse index at most~$I$.
The second fundamental form $|A_{F_n}|$ of a sequence $\{ F_n\}_n
\subset \Lambda$ may fail to be uniformly bounded, which 
leads to lack of compactness of $\Lambda$. Nevertheless,
the interesting ambient geometry of the immersions $F_n$ can be proven
to be well organized locally around at most $I$ points $p_{1,n},\ldots ,
p_{k,n}\in M_n$ ($k\leq I$) where $|A_{F_n}|$ takes on arbitrarily large 
local maximum values. 
Around any of these points $p_{i,n}$, one can perform a blow-up analysis and find a limit
of (a subsequence of) expansions $\l_nF_n$ of the $F_n$ (that is, 
we view $F_n$ as an immersion with constant mean curvature $H_n/\l_n$ in the
scaled ambient manifold $\l_nX_n$ for a sequence $\{ \l_n\}_n\subset \R^+$ tending to 
$\infty$). This limit is a complete immersed minimal surface
 $f\colon \Sigma\la \R^3$ with finite total curvature,
passing through the origin $\vec{0}\in \R^3$. 
Recall that such an $f$ has finitely many ends, each of which is 
a multi-valued graph of finite multiplicity (spinning) $s\in \N$, 
over the exterior of a disk in the tangent plane at infinite for $f$ at that end. 
Thus, arbitrarily small almost perfectly formed copies of large compact portions 
of $f(\Sigma)$ can be reproduced in $F_n(M_n)$ around $F_n(p_{i,n})$ 
for $n$ sufficiently large. Complete, finitely connected and finitely branched minimal surfaces with 
finite total curvature in $\R^3$ 
appear naturally when considering clustering 
phenomena in this framework:
It may occur that different blow-up limits of the $F_n$ around
$p_{i,n}$ at different scales $\l_{1,n}>\l_{2,n}$ with $\l_{1,n}/\l_{2,n}\to \infty$ as $n\to \infty$, produce different limits $f_j\colon \S_j\la \R^3$, $j=1,2$, with Index$(f_1)+\mbox{Index}(f_2)\leq I$; 
in this case, all the geometry of $f_1(\Sigma_1)$ collapses around $\vec{0}\in f_2(\S_2)$, and every end of $f_1(\Sigma_1)$
with multiplicity $m\geq 3$ produces a branch point at the origin for $f_2(\Sigma_2)$ of branching order $s-1$. For details 
about this clustering phenomenon and how to organize these blow-up limits
in hierarchies appearing around $\{ p_{i,n}\} _n$, see the paper~\cite{mpe18} by Meeks and the second author.

The main goal of this paper is to generalize the classical Henneberg 
minimal surface $H_1$ to an infinite family of connected, 1-sided, complete, 
finitely branched, stable minimal surfaces in $\R^3$. Branch points are unavoidable if we seek for
complete, non-flat stable minimal surfaces by the aforementioned results~\cite{cp1,fs1,po1,ros9};
1-sidedness is also necessary condition for stability (see 
Proposition~\ref{propos4} below). Our examples can be grouped into 
subfamilies depending on the number of branch points (this will be encoded by an integer $m\in \N$ called the {\it complexity}). The
most symmetric examples $H_m$ in each subfamily of complexity $m$
will be studied in depth (Section~\ref{raices}). 
Depending on the parity of $m$, either $H_m$
or its conjugate minimal surface $H_m^*$ (which does not gives rise to
a 1-sided surface, see Section~\ref{sec5.4}) 
can be viewed as the unique solution of 
a Bj\"{o}rling problem for a planar hypocycloid (Section~\ref{sec5.7}). 
The isometry group of $H_m$ is  isomorphic to the dihedral group $D_{2m+2}$ if $m$ is odd and to the
group $D_{m+1}\times \mathbb{Z}_2$ if $m$ is even (Section~\ref{sec5.8}).
We will also prove that $H_1$ is the only element in the subfamily
with complexity $m=1$ (Theorem~\ref{thm15}), while for $m\geq 2$, $H_m$ can be deformed in multiparameter families: Proposition~\ref{lema14}
gives an explicit 1-parameter family of examples with complexity
$m=2$, interpolating between $H_2$ and a limit which turns out to be $H_1$
(Section~\ref{sec6.2.1}), and the subfamily of examples with complexity
$m=2$ is a two-dimensional real analytic manifold around $H_2$ (Section~\ref{sec6.2.2}).

\section{$1$-sided branched stable minimal surfaces}
We start with the Weierstrass data $(g,\omega)$ on a
Riemann surface $\Sigma$, so that $(g,\omega)$ solves the
period problem and produces a conformal harmonic map $X\colon \Sigma \la \R^3$ given by the classical formula
\begin{equation}
	X=\text{Re}\int(\phi_1,\phi_2,\phi_3)=\text{Re}\int\left(\frac{1}{2}(1-g^2)\omega,\frac{i}{2}(1+g^2)\omega,g\omega\right).
	\label{1}
\end{equation}
We will assume that $X$ is an immersion outside of a locally finite set of points ${\mathcal B}\subset \Sigma$, where $X$ fails
to be an immersion (points of ${\mathcal B}$ are called
{\it branch points} of $X$). Such an $X$ will be called a {\it branched
minimal immersion.} The induced (possible branched) metric is given by
\begin{equation}
ds^2=\frac{1}{4}(1+|g|^2)^2|\omega|^2.
\label{ds2}
\end{equation}

The local structure of $X$ around a branch point in ${\mathcal B}$ is well-known, see e.g. Micallef and
White~\cite[Theorem~1.4]{miwh1} for details. Given $p\in
{\mathcal B}$, there exists a conformal coordinate $(D,z)$
for $\Sigma$ centered at $p$ (here $D$ is the closed unit
disk in the plane), a diffeomorphism $u$ of $D$ and a
rotation $\phi$ of $\R^3$ such that $\phi \circ X\circ u$
has the form
\[
z\mapsto (z^q,x(z))\in \C\times \R\sim \R^3
\]
for $z$ near $0$, where $q\in \N$, $q\geq 2$, $x$ is of class $C^2$, and $x(z)=o(|z|^q)$. 
In this setting, the {\it branching order of $p$} is defined to be $q-1\in \N$.

Let us assume that $X$ produces a $1$-sided branched minimal
surface; this means that there exists an anti-holomorphic
involution without fixed points $I:\Sigma\rightarrow\Sigma$
such that $I\circ \phi_j=\overline{\phi_j}$ for $j=1,2,3$.
This is equivalent to 
\begin{equation}\label{2.3}
-1/\overline{g}=g\circ I,\qquad 
I^*\omega=-\overline{g^2\omega}.
\end{equation}
In particular, $I$ must preserve the set ${\mathcal B}$. 
$\Sigma/\langle I\rangle$ is a non-orientable differentiable surface endowed with a conformal class of metrics,
and the harmonic map $X$ induces another harmonic map 
$\wh{X}\colon \Sigma/\langle I\rangle \la \R^3$
such that $\wh{X}\circ \pi =X$, where $\pi\colon \Sigma \to
\Sigma/\langle I\rangle$ is the natural projection ($\wh{X}$ is a branched
minimal immersion). Reciprocally, every $1$-sided conformal harmonic map 
can be constructed in this way.

\begin{remark}
	{\rm In the particular case that the compactification of $\Sigma$ is $\overline{\C}$, we can assume that $I(z)=-1/\overline{z}$
		and write $\omega=f\, dz$ globally. In this setting, the above equations give
		\begin{equation}\label{2.4}
			-1/\overline{g(z)}=g(-1/\overline{z}),\quad	f\circ I=-\overline{z^2 g^2 f}.
	\end{equation}}
\end{remark}

\begin{definition}
{\rm
Given a $1$-sided conformal harmonic map $\wh{X}\colon 
\Sigma/\langle I\rangle \la \R^3$, we denote by $\Delta$, 
$|A|^2$ the Laplacian and squared norm of the second 
fundamental form of $\wh{X}$. The {\it index} of $\wh{X}$ is defined
as the number of negative eigenvalues of the elliptic,
self-adjoint operator $L=\Delta+|A|^2$ (Jacobi operator of 
$X$) defined over the space of compactly supported smooth
functions $\phi \colon 	\Sigma \to \R$ such that $\phi \circ 
I=-\phi $. 
		$\wh{X}$ is 	said to be {\it stable} if its index is zero. 
		
		In the case
		$\wh{X}$ is finitely branched, the eigenvalues and eigenfunctions
		of the Jacobi operator of $X$ are well defined via
		a variational approach, since the codimension of the singularity set ${\mathcal B}$ is two (see~\cite{ty}), and
		stability also makes sense.}
\end{definition}

The next result is proven by Meeks and the second author in~\cite{mpe18}.
\begin{proposition}
\label{propos4}
Let $X\colon \Sigma \la \R^3$ be complete, non-flat, finitely branched
minimal immersion with branch locus
${\mathcal B}\subset \Sigma$. Then:
\begin{enumerate}
\item \cite[Proposition 3]{mpe18} 
If $X$ is stable, then $\Sigma $ is non-orientable and
$X({\mathcal B})$ contains more than 1 point.
\item \cite[Remark~3.6]{mpe18} Suppose that $\Sigma$ is non-orientable, $X$ has finite total
curvature and its extended unoriented Gauss map $G\colon \Pe^2
=\esf^2/\{ \pm 1\}\to \Pe^2$ 
is a diffeomorphism. Then, $X$ is stable.
\end{enumerate}
\end{proposition}

\section{The Bj\"{o}rling problem}
\label{Bjorling}
We next recall the basics of the classical Bj\"{o}rling problem, to be used later.
Let $\gamma \colon I\subset \R\rightarrow \R^3$ be an analytic regular curve and $\eta$ an analytic vector field along $\gamma$ 
such that $\langle \gamma(t),\eta(t)\rangle=0$ and $\|\eta(t)\|=1$ for all $t\in I$. 
The classical result due to  E.G. Bj\"{o}rling asserts that the following parametrization generates a minimal surface 
$S$ which contains $\gamma$ and has $\eta$ as unit normal vector along $\gamma$:
\[
X(u,v)=\text{Re}\left(\widetilde{\gamma}(w)- i\int_{w_0}^w \widetilde{\eta}(w)\times \widetilde{\gamma}'(w)\,dw\right),
\]
where $\widetilde{\gamma},\widetilde{\eta}$ are analytic extensions of the corresponding $\gamma,\eta$  and $w=u+i v$ 
is defined in a simply connected domain $\Omega\subset \C$ with $I\subset \Omega$. 
In particular, the surface $S$ is locally unique around $\gamma$ with this data (it is called the solution to the
Bj\"{o}rling problem with data $\g,\eta$).

In what follows, we will consider different Bj\"{o}rling problems for analytic planar curves $\g\subset 
\{z=0\}$ that fail to be 
regular at finitely many points. The above construction can be applied to each of the regular arcs of these curves
after removing the zeros of $\g'$. In all our applications, $\eta$ will be taken as
the (unit) normal vector field to $\g$ as a planar curve.

\section{The classical Henneberg surface}
The classical Henneberg minimal surface $H_1$ is the
$1$-sided, complete, stable minimal surface in $\R^3$ given by the Weierstrass data:
\begin{equation}
g(z)=z,\quad \omega=z^{-4}(z\pm i)(z\pm 1)dz=z^{-4}(z^4-1)dz, \quad z\in \overline{\C}-\{0,\infty\}.
\label{H1}
\end{equation}
$H_1$ has two branch points\footnote{Branch points of $H_1$
all have order 1 (locally the surface winds twice around
the branch point); this follows from direct computation, or
from Proposition 21 in White' s "Lectures on minimal
surfaces theory".} 
at $[1]=\{1,-1\},[i]=\{i,-i\}\in \mathbb{P}^2=\overline{\C}/\langle A\rangle $, where $A(z)=-1/\overline{z}$ is the antipodal map.
 By Proposition~\ref{propos4}, $H_1$ is stable.

$H_1$ can be
conformally parameterized (up to translations) by equation~\eqref{1}.
 After translating $X$ so that $X(e^{i\pi/4})=\vec{0}$,
the branch points of $H_1$ are mapped by $X$ to $(0,0,\pm 1)$ and a parametrization of $H_1$ in polar coordinates $z= re^{i\theta}$ is given by
\begin{equation}
X(r e^{i\theta})=
\left(\begin{array}{c}
	\frac{\cos\theta}{2}(r-\frac{1}{r})-\frac{ \cos (3\theta)}{6}(r^{3}-\frac{1}{r^{3}})
	\\
	-\frac{\sin \theta }{2}(r-\frac{1}{r})-\frac{ \sin (3\theta)}{6}(r^{3}-\frac{1}{r^{3}})
	\\
	\frac{\cos(2\theta)}{2}(r^{2}+\frac{1}{r^{2}})
\end{array}\right).
\label{H1X}
\end{equation}
Since $X(e^{i\theta})=(0,0,\cos(2\theta))$, then $X$ maps the unit circle into the vertical segment $\{(0,0,t)|\,t\in [-1,1]\}$. 
In this way,  $\theta\in [0,2\pi]\mapsto X(e^{i\theta})$ bounces between the two branch points of $H_1$ (observe that
the complement of this closed segment in the $x_3$-axis is not contained in $H_1$), see Figure~\ref{fig:test}.
\begin{figure}[h]
\centering
\begin{subfigure}{.5\textwidth}
\centering
\includegraphics[width=.8\linewidth]{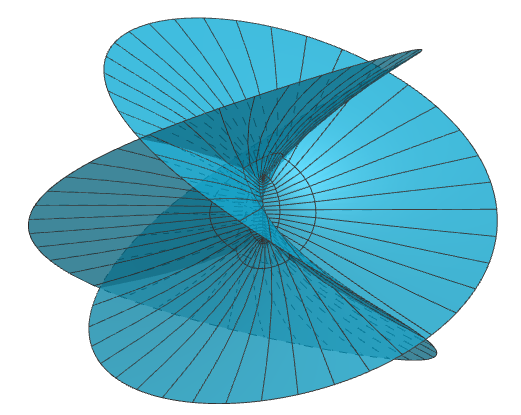}
\label{fig:sub1}
\end{subfigure}%
\begin{subfigure}{.5\textwidth}
\centering
\includegraphics[width=.8\linewidth]{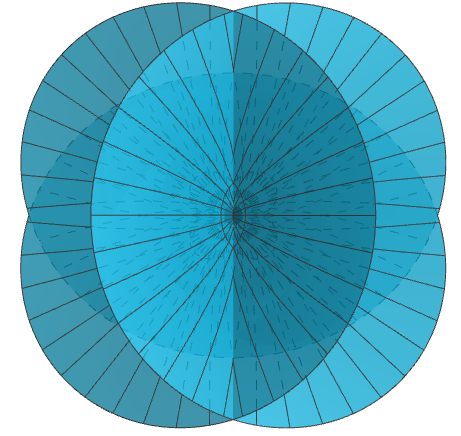}
\label{fig:sub2}
\end{subfigure}
\caption{The Henneberg surface $H_1$. After a traslation,
the branch points of $H_1$ are contained in the $x_3$-axis. $H_1$ contains two horizontal, orthogonal lines that bisect
the $x_1$- and $x_2$-axis. Left: Intersection of $H_1$ with
a ball of radius 8. Right: top view of $H_1$.}
\label{fig:test}
\end{figure}

\subsection{Isometries of $H_1$}
It is straightforward to check that
\begin{enumerate}

\item The antipodal map $A\colon \overline{\C}\to
\overline{\C}$ (in polar coordinates $(r,\theta)\mapsto (1/r,\pi +\theta)$) leaves the surface invariant. This is
the deck transformation, which is orientation reversing.

\item The map $z\mapsto -z$ (in polar coordinates $(r,\theta)\mapsto (r,\pi +\theta)$) induces the rotation by angle $\pi$ about the axis $x_3$ on the surface.

\item The inversion of the $z$-plane with respect to the unit circle, $z\mapsto 1/z$, (in polar coordinates $(r,\theta)\mapsto (1/r,\theta)$) is the composition of 
$A$ with $z\mapsto -z$, and thus, it also induces a rotation of angle $\pi$ about the $x_3$-axis on the surface.

\item The conjugation map $z\mapsto \overline{z}$ (in polar coordinates  $(r,\theta)\mapsto (r,-\theta)$) induces the reflection of $X$ about the plane $(x_1,x_3)$. 

\item The reflection about the imaginary axis (in polar
coordinates $(r,\theta)\mapsto (r,\pi-\theta)$) induces the 
reflection of $X$ about the plane $(x_2,x_3)$.

\item $X$ maps the half-line $\{r e^{-i\pi/4}\ | \ r\in (0,\infty)\}$ (respectively $\{r e^{i\pi/4}\ |\ r\in (0,\infty)\}$) injectively into $l_1=\text{Span}(1,1,0)$ (respectively $l_2=\text{Span}(1,-1,0)$). Thus, the rotations $R_1,R_2$ of angle $\pi$ about $l_1,l_2$ are isometries of $X$ ($R_1$ is induced by $z\mapsto -i\overline{z}$ and $R_2$ by $z\mapsto i\overline{z}$).

\item The map $z\mapsto iz$ (in polar coordinates $(r,\theta)\mapsto (r,\theta+\pi/2)$) induces the rotation of angle $\pi/2$ about the $x_3$-axis composed by a reflection in the $(x_1,x_2)$-plane.
\end{enumerate}

Together with the identity map, the above isometries form a
subgroup of the isometry group Iso$(H_1)$ of $H_1$, isomorphic to the dihedral group $D_4$. 
\begin{lemma}\label{lema5}
These are all the (intrinsic) isometries of $H_1$. 
\end{lemma}
\begin{proof}
This is a direct consequence of the fact that every intrinsic isometry $\phi$ of $H_1$ produces a
conformal diffeomorphism of $\C\setminus \{0\}$ into itself that preserves the set of
branch points of $H_1$. In particular, $\phi$ is of one of the aforementioned eight cases.
\end{proof}

\subsection{Associated family and the conjugate surface $H_1^*$.}
\label{rem6}
The flux vector of $H_1$ around the origin in $\C$
vanishes (in other words, the Weierstrass form $\Phi=(\phi_1,\phi_2,\phi_3)$ associated to $H_1$ is exact).
This implies that all associated surfaces $\{ \wt{H}_1(\varphi)\ | \ \varphi \in [0,2\pi)\} $ to the orientable
cover $\wt{H}_1=\wt{H}_1(0)$ of $H_1$ are well-defined as surfaces in $\R^3$ (the branched minimal immersion 
$\wt{H}_1(\varphi)$ has Weierstrass data $g_{\varphi}=g$, $\omega_{\varphi}=e^{i\varphi}\omega$ and it is 
isometric to $\wt{H}_1$, in particular it has the same branch locus as $\wt{H}_1$). 

None of the surfaces $\wt{H}_1(\varphi)$ except for $\varphi=0$ descends to the non-orientable quotient
$\Pe^2\setminus \{[0]\}$, because the second equation in~\eqref{2.3} is not 
preserved if we exchange $\omega $ by $e^{i\varphi}\omega $, 
$\varphi \in (0,2\pi)$. In particular, none of these associated surfaces are congruent to $H_1$.

The conjugate surface $H_1^*:=\wt{H}_1(\pi/2)$ is symmetric by reflection in the $(x_1,x_2)$-plane.
The intersection between $H_1^*$ and $\{ z=0\}$ consists of the {\it astroid} $\g_4$ parameterized by 
\[
t\mapsto \g_4(t)=\left(\begin{array}{c}
	-\sin(\theta)+\frac{\sin(3\theta)}{3}\\
	-\cos(\theta)-\frac{\cos (3\theta)}{3}  \\
	0
\end{array}\right),
\]
together with four rays starting at the cusps of the astroid in the
direction of their position vectors, see Figure~\ref{astroid}.
\begin{figure}[h]
	\centering
	\includegraphics[width=5cm]{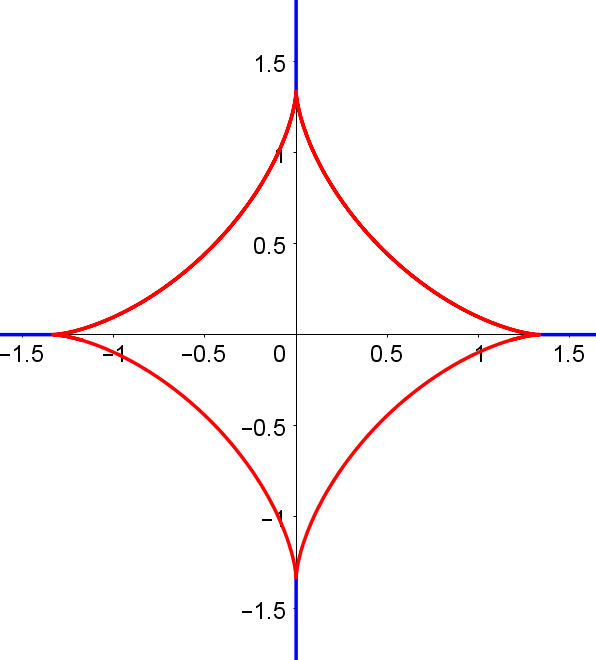}
	\caption{The astroid $\gamma_4$ (red) and the four rays obtained by intersecting $H_1^*$ with the $(x_1,x_2)$-plane (blue).}
\label{astroid}
\end{figure}

In particular, $H_1^*$ is the solution of the Bj\"{o}rling problem for the curve $\g_4$ and the choice of 
unit normal field the normal vector to $\g_4$ as a planar curve, see also Remark~\ref{rem8a} below.

\section{Generalized Henneberg surfaces}
We will next search for a $1$-sided, complete, stable minimal surface in $X\colon \Sigma \la \R^3$ with $\Sigma=\overline{\mathbb{C}}\setminus \mathcal{E}$, $\mathcal{E}$ finite and $g(z)=z$. Hence, $I(z)=-1/\overline{z}$, $\wh{X}=X/\langle I\rangle \colon \S/\langle I\rangle \la \R^3$ 
is stable and \eqref{2.4} writes
\begin{equation}\label{2.5}
f(-1/\overline{z})=-\overline{z^4 f(z)}.
\end{equation}

\subsection{General form for $f$}
We take a general rational function
\begin{equation}\label{2.6}
f(z)=\frac{c}{z^{m+3}}\frac{\prod_{j=1}^M(z-a_j)}{\prod_{j=1}^N(z-b_j)},
\end{equation}
where $c,a_j,b_j\in \C^*$, $m\in \mathbb{N}$, $M,N\in \N \cup \{0\}$ are to be determined. 
\begin{remark}
	\label{rem8}
{\rm 
\begin{enumerate}
\item Hennerberg’s surface $H_1$ has $f(z)=z^{-4}(z^4-1)$, hence $c=1$, $m=1$, $N=0$, $M=4$, $\{a_j\}=\{\pm 1,\pm i\}$. 
\item The zeros of the induced the metric~\eqref{ds2} (branch 
points of the surface) occur precisely at the points $a_j$; the ends occur at $0,\infty $ and at the points $b_j$ (in particular, both families $\{ a_j\}_j,\{ b_j\}_j$ must then
come in pairs of antipodal points, see also~\eqref{10}
below). 
\item A consequence of the last observation is that when the above 
rotations in $\R^3$ of our surfaces (provided that the Weierstrass data close periods) are not allowed unless the 
axis of rotation is vertical. 
\end{enumerate}
}
\end{remark}

Imposing \eqref{2.5} to \eqref{2.6} we get
\[
c(-1)^{m-1+M-N}\overline{z}^{3+m-M+N}\frac{\prod_{j=1}^M(1+a_j\overline{z})}{\prod_{j=1}^N(1+b_j \overline{z})}=f(-1/\overline{z})=-\overline{z^4 f(z)}=-\frac{\overline{c}}{\overline{z}^{m-1}}\frac{\prod_{j=1}^M (\overline{z}-\overline{a_j})}{\prod_{j=1}^N (\overline{z}-\overline{b_j})},
\]
thus
\begin{equation}\label{2.7}
\overline{c}(-1)^{m+M-N}z^{2+2m-M+N}
\prod_{j=1}^M(1+\overline{a_j}z)\prod_{j=1}^N(z-b_j)
=c\prod_{j=1}^M(z-a_j)\prod_{j=1}^N(1-\overline{b_j}z),
\end{equation}
from where we deduce that
\begin{equation}\label{2.8}
2+2m-M+N=0,
\end{equation}
in particular $M-N$ is even. Substituting $z=0$ in \eqref{2.7} we get
\begin{equation}\label{2.9}
\overline{c}(-1)^{m}\prod_{j=1}^N b_j=c\prod_{j=1}^M a_j.
\end{equation}
Using \eqref{2.9}, we can rewrite \eqref{2.7} as an equality between monic polynomials in $z$:
\[
\prod_{j=1}^M\left(\frac{1}{\overline{a_j}}+z\right)\prod_{j=1}^N(z-b_j)=\prod_{j=1}^M(z-a_j)\prod_{j=1}^N\left(\frac{1}{\overline{b_j}}+z\right),
\]
from where we deduce that 
\begin{equation}
\{a_1,\dots, a_M\}=\{-1/\overline{a_1},\dots,-1/\overline{a_M}\}, \quad \{b_1,\dots, b_N\}=\{-1/\overline{b_1},\dots,-1/\overline{b_N}\}.
\label{10}
\end{equation}
that is, $M,N$ are even, the $a_j$ (resp. $b_j$) are given by $M/2$ (resp. $N/2$) pairs of antipodal points in $\C^*$. 
Now \eqref{2.8} and \eqref{2.9} give respectively:
\begin{equation}\label{2.10}
1+m-\widetilde{M}+\widetilde{N}=0,
\end{equation}
\begin{equation}\label{2.11}
-\overline{c}\prod_{j=1}^{N/2} \frac{b_j}{\overline{b_j}}=c\prod_{j=1}^{M/2}\frac{a_j}{ \overline{a_j}}.
\end{equation}

\subsection{Solving the period problem in the one-ended case: complexity}
From \eqref{2.3} and \eqref{2.6} we see that the points where $ds^2$ can blow up are $z=0, b_1, \dots, b_N$ and its antipodal points. In order to keep the computations simple, we will assume there are no $b_j$'s, i.e. $N=0$ (or equivalently $M/2=m+1$), which reduces the period problem to imposing 
\[
\overline{\int_\gamma g^2 \omega}=\int_\gamma \omega,
\qquad \mbox{Re}\int_\gamma f\omega=0,
\]
where $\gamma=\{|z|=1\}$, or equivalently,
\begin{equation}\label{2.12}
\overline{\text{Res}_0(g^2 f)}=-\text{Res}_0(f),\quad \text{Im}\,\text{Res}_0(gf)=0.
\end{equation}
We can simplify \eqref{2.6} to
\begin{equation}\label{2.13}
f(z)=\frac{c}{z^{m+3}}\prod_{j=1}^{m+1}(z-a_j)\left(z+\frac{1}{\overline{a_j}}\right),
\end{equation}
which satisfies~\eqref{2.5} (this is the condition to
descend to the quotient as a 1-sided surface, provided that
the period problem~\eqref{2.12} is solved) if and only
if~\eqref{2.11} holds, which in this case reduces to
\begin{equation}\label{16}
	-\frac{\overline{c}}{c}=\prod_{j=1}^{m+1}\frac{a_j}{ \overline{a_j}}.
\end{equation}

We call 
\begin{equation}\label{2.14}
P(z):=\prod_{j=1}^{m+1}(z-a_j)\left(z+\frac{1}{\overline{a_j}}\right)=\sum_{h=0}^{2m+2} A_h z^h.
\end{equation}
Thus,
\[
\text{Res}_0(f)=c\, \text{Res}_0(\sum_{h=0}^{2m+2}A_{h} z^{h-m-3})=c A_{m+2},
\]
\[
\text{Res}_0(g^2f)=c\, \text{Res}_0(\sum_{h=0}^{2m+2}A_{h} z^{h-m-1})=c A_{m},
\]
\[
\text{Res}_0(gf)=c\, \text{Res}_0(\sum_{h=0}^{2m+2}A_{h} z^{h-m-2})=c A_{m+1}.
\]
Thus, \eqref{2.12} reduces to 
\begin{equation}\label{2.15}
\overline{c A_{m}}=-c A_{m+2},\quad \text{Im}(c A_{m+1})=0.
\end{equation}
\begin{remark}
{\rm
We can assume $|c|=1$ due
to the fact that multiplying the Weierstrass form by a positive number does not affect to solving the period 
problem and just multiplies the resulting surface by a homothety. Similarly, exchanging $c$ by $-c$ 
does not affect to solving the period problem.
}
\end{remark}
We also write $a_j=|a_j|e^{i\theta_j}$, $\theta_j\in \R$. Thus,
\[
-a_j+\frac{1}{\overline{a_j}}=(-|a_j|+\frac{1}{|a_j|})e^{i\theta_j}, \quad\frac{a_j}{\overline{a_j}}=e^{2i\theta_j},
\] 
and so, 
\begin{equation}\label{2.16}
P(z)=\prod_{j=1}^{m+1}\left(z^2+(-|a_j|+\frac{1}{|a_j|})e^{i\theta_j}z-e^{2i\theta_j}\right)
\end{equation}
\begin{definition}
	\label{def7}
Given $m\in \N$, a list $(c,a_1,\ldots ,a_{m+1})\in \esf^1\times (\C^*)^{m+1}$ 
solving the equations \eqref{16},\eqref{2.15} will be called a {\it solution of the period problem with 
complexity $m$.} Note that geometrically,
$a_1,\ldots,a_{m+1}$ are the Gaussian images of the branch
points of the resulting surface.
\end{definition}

\subsection{The case when the $a_j$ are the $(2m+2)$-roots of unity}
\label{raices}
For each complexity $m$, there is a most symmetric configuration that gives rise to a solution of the period problem
for that complexity, which we describe next.

Take the $a_j$ as the solutions of the equation $a^{2m+2}=1$ (i.e. $|a_j|=1$ and $\theta_j=\frac{\pi}{m+1}(j-1)$, $j=1,\dots,m+1$).
Observe that
\[
\prod_{j=1}^{m+1}\frac{a_j}{ \overline{a_j}}=\prod_{j=1}^{m+1}e^{2i\theta_j}=e^{2i\sum_{j=1}^{m+1}\t_j}=e^{\frac{2\pi i}{m+1}\sum_{j=1}^{m+1}(j-1)}=
e^{\frac{2\pi i}{m+1}\frac{m(m+1)}{2}}=e^{i\pi m},
\]
hence the validity of~\eqref{16} is equivalent in this case
to 
\begin{equation}
c=\pm i^{m-1}.
\end{equation}
As for equation~\eqref{2.15}, note that \eqref{2.16} can be written as
\[
P(z)=\prod_{j=1}^{m+1}(z^2-e^{2i\theta_j})=z^{2m+2}-1,
\]
and thus $A_{m}=0$ (because $m>0$), $A_{m+2}=0$ (because $m+2<2m+2$) and $A_{m+1}=0$. In particular, \eqref{2.15} is trivially satisfied for each value of $c\in \C^*$. Therefore, the Weierstrass data
\begin{equation}
g(z)=z, \quad \omega=i^{m-1}z^{-m-3}(z^{2m+2}-1)dz,\qquad z\in \C^*,
\label{WHm}
\end{equation}
give rise to a $1$-sided, complete, stable minimal surface 
$H_{m}$. For $m=1$ we recover the classical Henneberg’s surface.

\subsection{Associated family and the conjugate surface $H_m^*$.}
\label{sec5.4}
Since $A_m=A_{m+1}=A_{m+2}=0$, the flux vector of
$H_m$ around the origin in $\C$ vanishes and the Weierstrass form $\Phi=(\phi_1,\phi_2,\phi_3)$ 
associated to $H_m$ is exact. Thus all associated surfaces  $\{ \wt{H}_m(\varphi)\ | \ \varphi \in [0,2\pi)\} $
to the orientable cover $\wt{H}_m=\wt{H}_m(0)$ of
are well-defined. As in the case $m=1$ (see Section~\ref{rem6}), none of these associated
surfaces descends to the $1$-sided quotient, except for $\pm H_m$. Let $H_m^*:=\wt{H}_m(\pi/2)$ be 
the conjugate surface to $H_m$. 

The behavior of $H_m$ is very different depending on the parity of $m$.
A naive justification of this dependence on the parity of $m$ comes from
the fact that the coefficient for $\omega $ changes from $\pm 1$ for $m$ odd to $\pm i$ for $m$ even.
A more geometric interpretation of this dependence will be given next.

\subsection{The case $m$ odd}
If $m\in \N$ is odd,
\eqref{WHm} gives $\omega=z^{-m-3}(z^{2m+2}-1)dz$. 
Although $H_m$ has $m+1$ branch points in $\S=\Pe^2\setminus \{[0]\}$ 
(the classes of the $(2m+2)$-roots of unity under the antipodal map), 
they are mapped into just two different points in $\R^3$: 
after translating the surface in $\R^3$ so that 
$X(e^{i\frac{\pi}{2(m+1)}})=\vec{0}$ (we are using the notation 
in~\eqref{1}), the branch points of $H_m$ are mapped to $(0,0,\pm 1)$ and a 
parameterization of $H_m$ in polar coordinates is (compare with~\eqref{H1X})
\begin{equation}
\label{Hmimpar}
X(r e^{i\theta})=
\left(\begin{array}{c}
x_1\\ x_2\\ x_3
\end{array}\right)
=\left(\begin{array}{c}
\frac{\cos (m\t)}{2 m}(r^m-\frac{1}{r^m})-\frac{ \cos ((m+2)\theta)}
{2 (m+2)}(r^{m+2}-\frac{1}{r^{m+2}})
\\
-\frac{\sin (m\t)}{2 m}(r^m-\frac{1}{r^m})-\frac{ \sin ((m+2)\theta)}
{2 (m+2)}(r^{m+2}-\frac{1}{r^{m+2}})
\\
\frac{\cos((m+1)\theta)}{m+1}(r^{m+1}+\frac{1}{r^{m+1}})
\end{array}\right).
\end{equation}
$X$ maps the unit circle into the vertical segment $\{(0,0,t)|\,t\in 
[-1,1]\}$. $\theta\in [0,2\pi]\mapsto
X(e^{i\theta})$ bounces between the two branch points of $H_m$, and the 
complement of this closed segment in the $x_3$-axis is not contained in 
$H_m$. $H_m\cap \{ x_3=0\}$ consists of an equiangular system of 
$m+1$ straight lines passing through the origin (the images by $X$ of the 
straight lines of arguments	$\t=\frac{\pi/2+k\pi}{m+1}$, $k=0,\ldots,m$ in 
polar coordinates), see Figure~\ref{fig:H2H3} right for $H_3$.

\subsection{The case $m$ even}
If $m$ is even (and non-zero), 
\eqref{WHm} produces $\omega=i\, z^{-m-3}(z^{2m+2}-1)dz$. In this case, a 
parametrization of	$H_m$ in polar coordinates is 
\begin{equation}
\label{Hmpar}
X(r e^{i\theta})=
\left(\begin{array}{c}
x_1\\ x_2\\ x_3
\end{array}\right)
=\left(\begin{array}{c}
-\frac{ \sin (m\theta)}{2 m}\left(r^m+\frac{1}{r^m}\right)
+\frac{\sin ((m+2)\theta)}{2 (m+2)}
\left(r^{m+2}+\frac{1}{r^{m+2}}\right)
\\
-\frac{ \cos (m\t)}{2 m}\left(r^m+\frac{1}{r^m}\right)
-\frac{ \cos ((m+2)\theta  )}{2 (m+2)}\left(r^{m+2}+\frac{1}{r^{m+2}}\right)
\\
\frac{\sin((m+1)\theta)}{m+1}(\frac{1}{r^{m+1}}-r^{m+1})
\end{array}\right).
\end{equation}
$X$ maps the unit circle $\{r=1\}$ into a certain hypocycloid contained
in the plane $\{ x_3=0\}$, as we will explain next.

A {\it hypocycloid} of inner radius $r>0$ and outer radius $R>r$ 
is the planar curve traced by a point on a circumference of 
radius $r$ which is rolling along the interior of another circumference
(which is fixed) of radius $R$. It can be parametrized by $\alpha(t)=(x(t),y(t))$,
$t\in \R$, where
\[
x(t)=-(R-r)\sin t +r\sin\left(\frac{R-r}{r}t\right),\qquad 
y(t)=-(R-r)\cos t -r \cos\left(\frac{R-r}{r}t\right).
\]

Using \eqref{Hmpar}, we deduce that the image by $X$ of the unit circle $\esf^1\subset \C$ has the following parametrization:
\begin{equation}
\label{hypoc}
\theta\in [0,2\pi)\mapsto X(e^{i\theta})=\left(\begin{array}{c}
	-\frac{\sin(m\theta)}{m}+\frac{\sin((m+2)\theta)}{m+2}\\
	-\frac{\cos(m \theta)}{m}-\frac{\cos(m+2)\theta}{m+2}  \\
	0
\end{array}\right).
\end{equation}
From~\eqref{hypoc} we deduce that, up to the reparametrization $t=m\theta$, 
$X(\esf^1)$ is the hypocycloid of inner radius $r=\frac{1}{m+2}$ and outer
radius $R=\frac{2m+2}{m(m+2)}$, which has exactly $m+1$ cusps. 
These cusp points are the images by $X$ of the 
$m+1$ branch points of $H_m$. In particular, $H_m$ is the unique minimal 
surface obtained as solution of the Bj\"{o}rling problem for the hypocycloid of $m+1$ 
cusps (this number of cusps is any odd positive integer, at least three), 
inner radius $r=\frac{1}{m+2}$ and outer radius 
$R=\frac{2m+2}{m(m+2)}$, when we take as normal vector field $\eta$
(see Section~\ref{Bjorling} for the notation) the normal vector to the
hypocycloid as a planar curve. 

We depict this planar curve in the
simplest cases $m=2,4,6$ in Figure \ref{planar} in red.
\begin{figure}[h]
	\centering
	\includegraphics[width=4.9cm]{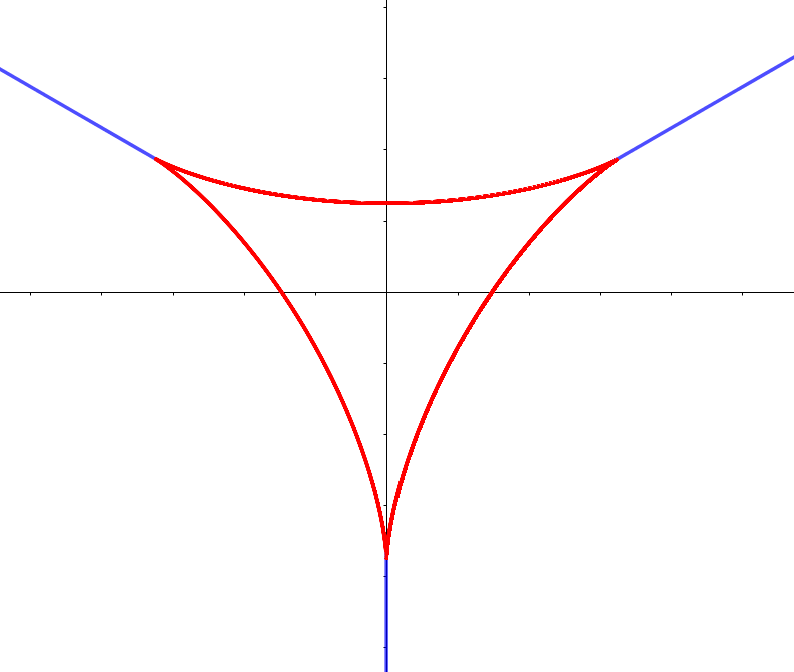}
	\includegraphics[width=4.9cm]{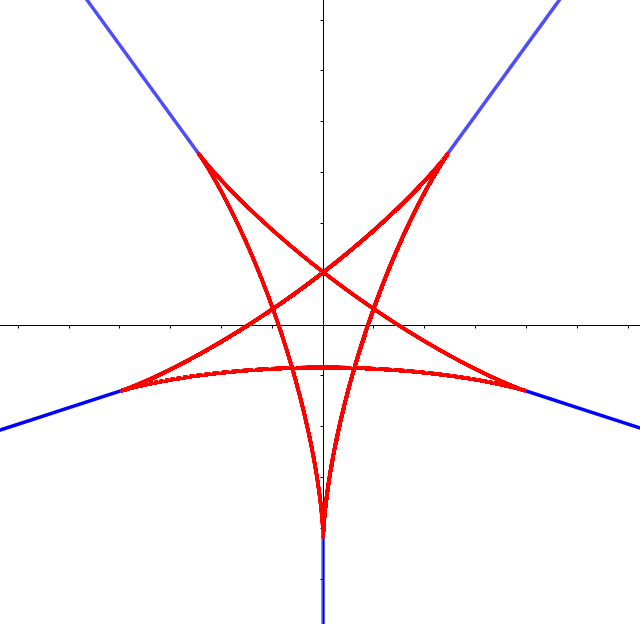}
	\includegraphics[width=4.9cm]{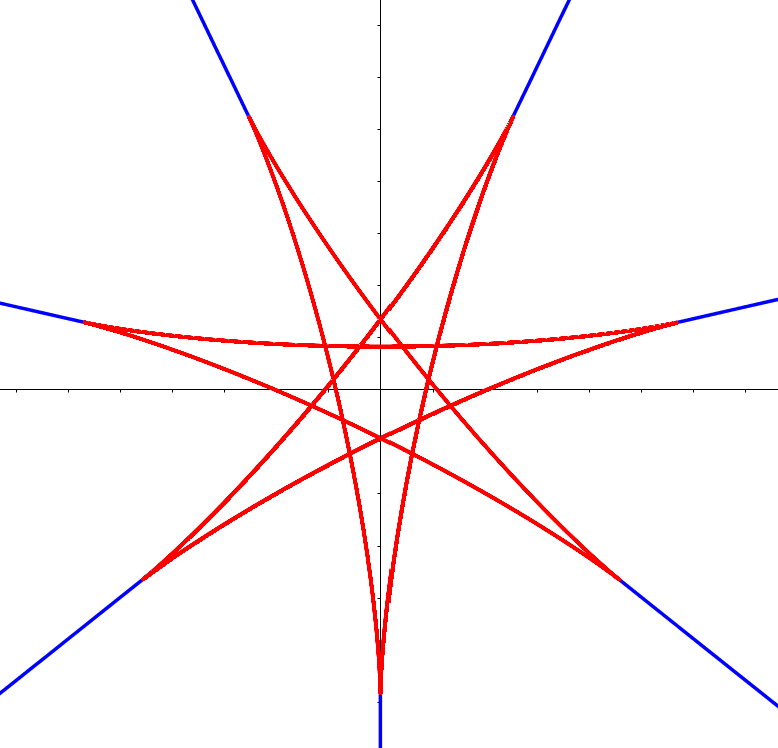}
	\caption{The intersection of $H_m$ (with $m>0$ even) with $\{x_3=0\}$ consists of a hypocycloid with $m+1$ cusps (in red) together with  half-lines $\{ tp\ | \ t\geq 1\}$ that start from each of these cusp points $p$. Left: $H_2\cap \{ x_3=0\}$, where
	the branch points have coordinates $(0,-\frac{3}{4},0),
	(-\frac{3 \sqrt{3}}{8},	\frac{3}{8},0),(\frac{3 \sqrt{3}}{8},
	\frac{3}{8},0)$. Center: $H_4\cap \{ x_3=0\}$, Right: $H_6\cap \{ x_3=0\}$.
	\label{planar}
	}
	\label{fig:planar}
\end{figure}
\begin{figure}[h]
	\centering
	\includegraphics[width=8cm]{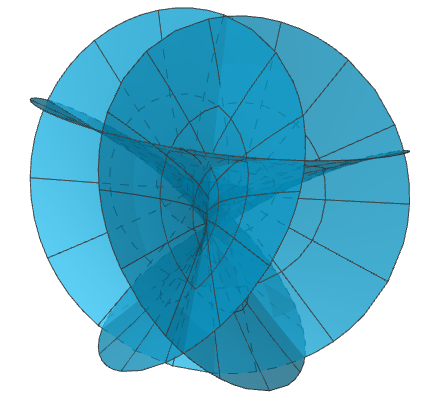}
	\includegraphics[width=8cm]{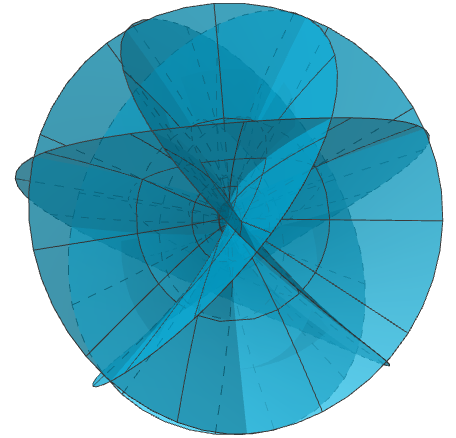}
	\caption{Left: $H_2$. Right: $H_3$.}
\label{fig:H2H3}
\end{figure}

\subsection{Revisiting the case $m$ odd: $H_m^*$ as a solution of 
a Bj\"{o}rling problem for a hypocycloid}
\label{sec5.7}

Using the Weierstrass formula~\eqref{1}, it can be easily seen that
the conjugate surface $H_m^*$ of $H_m$ with odd $m$ can be parameterized
in polar coordinates $z=re^{i\t}$ by $X^*(re^{i\t})$ given by the same formula as the right-hand-side of~\eqref{Hmpar}. $X^*(\esf^1)$ parameterizes a
hypocycloid $\g_{2m+2}$ with inner radius 
$r=\frac{1}{m+2}$ and outer radius $R=\frac{2m+2}{m(m+2)}$. Since 
\[
\frac{R}{r}=\frac{2m+2}{m},
\] 
we deduce that $\g_{2m+2}$ has $2m+2$ cusps\footnote{For a hypocycloid of 
inner radius $r>0$ and outer radius $R>r$, the quotient 
$R/r$ expresses the number of times that the inner
circumference rolls along the outer circumference until it completes a 
loop. If $R/r$ is a rational number and $a/b$ is the irreducible fraction 
of $R/r$, then $b\cdot a/b=a$ counts the number of times that the inner 
circumference rolls until the point that generates the hypocycloid reaches 
its initial position. This number $a$ coincides with the number of 
cusps.}.
Observe that $2m+2$ is a positive multiple of $4$ because $m$ is odd;
and conversely, every positive multiple of $4$ can be written as 
$2m+2$ for a unique $m\in \N$ odd. This tells us that 
for any $m\in \N$ odd, $H^*_m$ is the unique solution to the Björling
problem for the hypocycloid $\g_{2m+2}$,
when we take as normal vector field $\eta$ the normal vector to $\g_{2m+2}$
as a planar curve.

\begin{remark}
\label{rem8a}
{\rm 
\begin{enumerate}
\item In the particular case of a hypocycloid of 4 cusps (called {\it astroid}),
we recover the conjugate surface $H_1^*$ of the classical Henneberg surface.
This result was described by Odehnal~\cite{ode1}, who also studied the
 Björling problem for an hypocycloid $\g_3$ of three cusps 
from the viewpoint of algebraic surfaces. 
\item We have described the minimal surfaces obtained as the
solution of a Björling problem over a hypocycloid if the number of its
cusps is either any given odd number or a multiple of four. The case
that remains is when the hypocycloid has $4k+2$ cusps, $k\in \N$. The
corresponding solution to this Björling problem 
can be also explicitly described by the 
parametrization~\eqref{Hmpar}, now with a parameter $m\in \Q$. 
Namely, if we choose $m$ to be 
of the form $m=\frac{1}{2k}$, $k\in \N$, inner radius $r=\frac{1}{m+2}$
and outer radius $R=\frac{2m+2}{m(m+2)}$, then 
\[
\frac{R}{r}=\frac{2m+2}{m}=4k+2,
\]
which ensures that the complete branched minimal surface $H_{\frac{1}{2k}}=X(\C\setminus \{ 0,\infty\})$ (here $X$ is given 
by~\eqref{Hmpar}) is symmetric by reflection in the $(x_1,x_2)$-plane), 
and $X(\esf^1)$ is a hypocycloid with $4k+2$ cusps. $H_{\frac{1}{2k}}$
does not descend to a 1-sided quotient. 
\end{enumerate}
}
\end{remark}

\subsection{Isometries of $H_m$}
\label{sec5.8}
As expected, the  isometry group of $H_m$
depends on whether $m$ is even or odd. 
\par
\vspace{.2cm}
Suppose firstly that \underline{$m$ is odd.} In this case, \eqref{Hmimpar}
gives:
\begin{enumerate}[(O1)]
\item The reflection of the $z$-plane about the imaginary axis, $re^{i\theta} \mapsto re^{i(\pi-\theta)}$, produces
via $X$ the reflectional symmetry about the $(x_2, x_3)$-plane in $H_m$.
    
\item The rotation $re^{i\theta} \mapsto re^{i(\t +\pi+\frac{\pi}{m+1})}$
of angle $\pi+\frac{\pi}{m+1}$ about the origin in the 
$z$-plane, gives that $H_m$ is symmetric under the rotation of angle
$\frac{\pi}{m+1}$ about the $x_3$-axis composed by a reflection in the 
$(x_1,x_2)$-plane.
\end{enumerate} 
(O1), (O2) generate a subgroup of the extrinsic isometry group $\text{Iso}(H_m)$ of $H_m$, isomorphic to the dihedral group $D_{2m+2}$.
\par
\vspace{.2cm}
Now assume that \underline{$m$ is even.} Using~\eqref{Hmpar}, we obtain:
\begin{enumerate}[(E1)]
\item The reflection $re^{i\theta} \mapsto re^{i(\pi-\theta)}$
of the $z$-plane about the imaginary axis produces
via $X$ the reflectional symmetry about the $(x_2, x_3)$-plane in $H_m$ (this is a common feature of both the odd and even cases).
    
\item The rotation $re^{i\theta} \mapsto re^{i(\t +\frac{2\pi}{m+1})}$
of angle $\frac{2\pi}{m+1}$ about the origin in the 
$z$-plane, gives that $H_m$ is symmetric under the rotation of angle
$\frac{2\pi}{m+1}$.

\item The antipodal map $re^{i\theta} \mapsto re^{i(\t +\pi)}$ in the 
$z$-plane, produces a reflectional symmetry of $H_m$ with respect to the
$(x_1,x_2)$-plane.
\end{enumerate} 
(E1), (E2), (E3) generate a subgroup of $\text{Iso}(H_m)$ isomorphic to the
group $D_{m+1}\times \mathbb{Z}_2$.

Repeating the argument in the proof of Lemma~\ref{lema5}, we now deduce the following.
\begin{lemma}
Regardless of the parity of $m$, these are all the (intrinsic) isometries of $H_m$.
\end{lemma}

\section{Moduli spaces of examples with a given complexity}
Our next goal is to analyze the structure of the family of
solutions of the period problem with a given complexity in 
the sense of Definition~\ref{def7}. For $m=1$, we will obtain uniqueness
of the Henneberg surface $H_1$. This uniqueness is a special feature of
the case $m=1$, since continuous families of examples for complexities
$m\geq 2$ can be produced.

We define the function $R\colon (0,\infty)\to (0,\infty)$, $R(r)=r-\frac{1}{r}$.
\subsection{Solutions with complexity $m=1$}

Since $m=1$, solving the period problem \eqref{2.15}
descending to the 1-sided quotient reduces to solving 
\begin{equation}\label{k0}
\overline{cA_{1}}=-cA_{3},\quad \text{Im}(cA_{2})=0,\qquad
-\frac{\overline{c}}{c}=\frac{a_1}{ \overline{a_1}}\frac{a_2}{ \overline{a_2}}.
\end{equation}
Suppose that a list $(c,a_1,a_2)\in \esf^1\times (\C^*)^2$
is a solution of the 1-sided period problem, with associated branched minimal
immersion $X$. Recall that $g(z)=z$ is its Gauss map. 
The list that gives rise to $H_1$ (Henneberg) is $(\pm 1,1,i)$.
\begin{remark}
{\rm 
Since rotations of our surfaces are not allowed unless the
rotation axis is vertical (see Remark~\ref{rem8}), we
we can assume $a_1\in \R^+$ from now on, although we cannot assume $a_1=1$.
}
\end{remark} 
 Write $a_1,a_2$ in polar coordinates as $a_1=r_1$, $a_2
=r_2 e^{i\theta_2}$, $r_1,r_2>0$, $\t_2\in [0,2\pi)$. 
\eqref{2.14} can be written as 
\begin{eqnarray}
	P(z)&=&z^4-\left[ R(r_1)+R(r_2)e^{i\t_2}\right]z^3
	-\left[ 1+e^{2i\theta_2}-R(r_1)R(r_2)e^{i\t_2}\right] z^2\nonumber
	\\
	& & 
	+\left[ R(r_1)e^{2i\t_2}+R(r_2)e^{i\t_2}
	\right] z
	+e^{2i\t_2},\nonumber
\end{eqnarray}
hence
\begin{eqnarray}
	A_1&=&R(r_1)e^{2i\t_2}+R(r_2)e^{i\t_2},
	\label{A1}
	\\
	A_2&=&-\left[ 1+e^{2i\theta_2}-R(r_1)R(r_2)e^{i\t_2}\right],\label{A2}
	\\
	A_3&=&-\left[ R(r_1)+R(r_2)e^{i\t_2}\right].
	\label{A3}
\end{eqnarray}
Writing $c=e^{i\be}$, we have 
\begin{eqnarray}
\overline{cA_{1}}+cA_{3}
&=&R(r_1)\left[ e^{-i(\be +2\t_2)}-e^{i\be }\right] +R(r_2)\left[ e^{-i(\be +\t_2)}
-e^{i(\be+\t_2)}\right] \nonumber
\\
&=&R(r_1)e^{-i\t_2}\left[ e^{-i(\be +\t_2)}
-e^{i(\be+\t_2)}\right] 
-2R(r_2)\sinh (i(\be +\t_2))\nonumber
\\
&=&-2e^{-i\t_2}R(r_1)\sinh (i(\be +\t_2)) -2iR(r_2)\sin (\be +\t_2)\nonumber
\\
&=& -2i\left[ R(r_1) e^{-i\t_2}+R(r_2)\right]\sin (\be+\theta_2),\label{18}
	\\ 
cA_{2}&=&-e^{i\be}\left( 1+e^{2i\theta_2}\right)
+R(r_1)R(r_2)e^{i(\be+\t_2)}\nonumber
\\
&=&-e^{i(\be+\t_2)}\left( e^{-i\t_2}+e^{i\t_2}\right)
+R(r_1)R(r_2)e^{i(\be+\t_2)}\nonumber
\\
&=& -\left[ 2\cosh (i\t_2)-R(r_1)
R(r_2)\right] e^{i(\be +\t_2)}\nonumber
\\
&=&-\left[ 2\cos \t_2
	-R(r_1)R(r_2)\right] e^{i(\be +\t_2)}.
	\label{19}
\end{eqnarray}
A list $(c,a_1,a_2)$ solves the period problem if and only if the right-hand-side of~\eqref{18} vanishes and the right-hand-side of~\eqref{19} is real.

The
third equation in~\eqref{k0} reduces to 
\begin{equation}
e^{2i(\be+\t_2)}=-1.
\label{25}
\end{equation}

\begin{theorem}
\label{thm15}
The Henneberg surface $H_1$ is the only surface with $m=1$ that solves the period problem and descends to a 1-sided quotient.
\end{theorem}
\begin{proof}
By the above arguments, the right-hand-side of~\eqref{18} vanishes, the right-hand-side of~\eqref{19} is real and~\eqref{25} holds.

\eqref{25} implies that $\sin(\be+\t_2)=\pm 1$. Since
the right-hand-side of~\eqref{18} vanishes, we have 
\begin{equation}
R(r_1) e^{-i\t_2}+R(r_2)=0.
\label{26}
\end{equation}
We have two possibilities:
\begin{itemize}
\item $r_1=1$. Thus \eqref{26} implies $r_2=1$. From,
\eqref{25} we have $\be+\t_2\equiv \pi/2 \text{ mod }\pi$ and from \eqref{19} we have $\cos \t_2=0$, thus $\t_2=\pi/2$ or $\t_2=3\pi /2$. This gives the lists $(1,1,i)$, $(-1,1,i)$, $(1,1,-i)$ and $(-1,1,-i)$. All of them give raise to the Henneberg surface.
	
\item $r_1\neq 1$. This implies $e^{-i\t_2}=-\frac{R(r_2)}{R(r_1)}$,
which is real. Hence $e^{-i\t_2}=\pm 1$. As the function
$r\mapsto R(r)$ is injective, this implies
$r_1=r_2$ and $\t_2=\pi$ or $r_2=1/r_1$ and $\t_2=0$. Since the right-hand-side of~\eqref{19} is real and~\eqref{25} holds, 
$2\cos \t_2-R(r_1)R(r_2)=0$. But in both cases $2\cos \t_2-R(r_1)R(r_2)$ does not vanish. Hence this possibility
cannot occur.
\end{itemize}
\end{proof}

\subsection{Solutions with complexity $m=2$}
Suppose that a list $(c=e^{i\be},a_1=r_1,a_2=r_2e^{i\t_2},a_3=r_3e^{i\t_3})
\in \esf^1\times \R^+\times (\C^*)^2$ is a solution
of the period problem with 1-sided quotient and associated branched minimal immersion $X$. 
The list that gives rise to $H_2$ is $(\pm i,1,e^{i\pi/3},e^{2i\pi/3})$.

Solving the period problem with 1-sided quotient is
equivalent to solving 
\begin{equation}
\overline{cA_{2}}=-cA_{4},\quad \text{Im}(cA_{3})=0,
\quad
-\frac{\overline{c}}{c}=\frac{a_2}{\overline{a_2}}\frac{a_3}{\overline{a_3}}
\label{k0m2}
\end{equation}
The third equation in~\eqref{k0m2} reduces to 
\begin{equation}
e^{2i(\be+\t_2+\t_3)}=-1.
\label{28}
\end{equation}
\eqref{2.14} can be written as 
\[
P(z)=z^6+A_5z^5+A_4z^4+A_3z^3+A_2z^2+A_1z+A_0,
\]
where 
\begin{eqnarray}
A_2&=&e^{2i(\t_2+\t_3)}+e^{2i\t_2}+e^{2i\t_3}
-R(r_1)R(r_2)e^{i(\t_2+2\t_3)}
-R(r_1)R(r_3)e^{i(2\t_2+\t_3)}
\nonumber
\\
& & 
-R(r_2)R(r_3)e^{i(\t_2+\t_3)},
\label{A32}
\\
A_3&=&\mbox{}\hspace{-.2cm}
2\textstyle{ \left[ R(r_2)\cos \t_3
+R(r_3)\cos \t_2
+R(r_1)\cos(\t_2-\t_3)
-\frac{1}{2}R(r_1)R(r_2)R(r_3)
\right]} e^{i (\t_2+\t_3)},\quad \label{A33}
\\
A_4&=&-(1+e^{2i\t_2}+e^{2i\t_3})
+R(r_1)R(r_2)e^{i\t_2}+R(r_1)R(r_3)e^{i\t_3}
+R(r_2)R(r_3)e^{i(\t_2+\t_3)}.
\label{A34}
\end{eqnarray}
Thus,
\begin{eqnarray}
\overline{cA_{2}}+cA_{4}
&=&2 e^{-i [\beta +2 (\t_2+\t_3)]} F,
\label{29}
\\
cA_3&=&\pm 2i G 
\label{30}
\end{eqnarray}
where 
\begin{eqnarray}
F&=&\textstyle{e^{2 i \t_3}+\left[2 \cos\t _2-R(r_1) R(r_2)\right]e^{i\t_2} 
	-R(r_3)\left[ R(r_1)+ R(r_2)e^{i\t_2}\right]e^{i \t_3},}\label{Fa}
\\
G&=&
\textstyle{
R(r_2)\cos \t_3
+R(r_3)\cos \t_2
+R(r_1)\cos(\t_2-\t_3)
-\frac{1}{2}R(r_1)R(r_2)R(r_3).
}
\label{G}
\end{eqnarray}

\begin{remark}\label{rem12a}
{\rm 
\begin{enumerate}[(I)]
\item From \eqref{G} we deduce that $G$ is real, hence the condition
$\text{Im}(cA_{3})=0$ only holds if and only if $G=0$.
We deduce that a list $(c,a_1,a_2,a_3)$ solves the 1-sided period problem 
if and only if~\eqref{28} holds and $F=G=0$.

\item The expression~\eqref{Fa} is symmetric in 
$(r_2,\t_2),(r_3,\t_3)$. This can be deduced from the 
symmetry of $A_2,A_4$, or directly checked by using the equality 
\begin{equation}
e^{2i\t_j}=2\cos \t_je^{i\t_j}-1,
\label{35}
\end{equation}
which transforms \eqref{Fa}  into 
		\begin{equation}
			F=(1+e^{2i\t_2}+e^{2i\t_3})-R(r_1)
			\sum_{j=2}^3R(r_j)e^{i\t_j}-R(r_2)R(r_3)e^{i(\t_2+\t_3)}.\label{31}
		\end{equation}
\end{enumerate}
}
\end{remark}

\begin{lemma}
	\label{lem12}
If $F=0$, then the coefficient of $R(r_1)$
in~\eqref{31} is non-zero. 
\end{lemma}
\begin{proof}
Suppose
$R(r_2)e^{i\t_2}+R(r_3)e^{i\t_3}=0$. This leads to one of 
the following two possibilities:  (a) $e^{i\t_2}=e^{i\t_3}$
and $R(r_2)=-R(r_3)$ or else (b) $e^{i\t_2}=-e^{i\t_3}$ and 
$R(r_2)=R(r_3)$. (a) implies $r_3=1/r_2$ and thus,
\eqref{31} gives $F=1 + e^{2i \t_2}(\frac{1}{r_2^2}+r_2^2)$.
(b) implies $r_2=r_3$ and \eqref{31} gives the same 
expression for $F$. In any case, we deduce from $F=0$ 
that $e^{2i\t_2}$ is real negative, hence 
$\frac{r_2^2}{r_2^4+1}=-e^{2i\t_2}=1$. This is impossible, 
since the function $x>0\mapsto \frac{x}{1 + x^2}$ has a 
unique maximum at $x=1$ with value $1/2$. 
\end{proof}

The next result describes a one-parameter family of non-trivial examples 
of complexity $m=2$ different from $H_2$.
\begin{proposition}
	\label{lema14}
Suppose that a list $(c,a_1,a_2,a_3)$ solves the 1-sided 
period problem. Then:
\begin{enumerate}
\item If $r_1=1$, and at least one of $r_2$ or $r_3$ equals
one, then $(c,a_1,a_2,a_3)=(\pm i,1,e^{i\pi/3},
e^{2i\pi/3})$ and the example is $H_2$.
\item If $\t_2+\t_3=0$ (mod $\pi$), then $r_2=r_3$ or $r_2=1/r_3$ and  $(r_1,r_2)$ are given by the following 
functions of $\t_2\in (\frac{\pi}{4},\frac{\pi }{3}]\cup [\frac{2 \pi }{3},\frac{3 \pi }{4}):$
\begin{eqnarray}
R(r_1(\t_2))&=&
\frac{1}{8 \sqrt{2}}\frac{\sqrt{f(\t_2)-3}}
{\cos \t_2 \cos (2 \t_2)} [f(\t_2)+3+4 \cos (2 \t_2)],
\label{40b}
\\
R(r_2(\t_2))&=&-\frac{\sqrt{f(\t_2)-3}}{\sqrt{2}},
\label{41b}
\end{eqnarray}
or else $(r_1,r_2)$ are given by the opposite expressions for both $R(r_1(\t_2)),R(r_2(\t_2))$, which exchange $(r_1,r_2)$ by $(\frac{1}{r_1},\frac{1}{r_2})$. Here,
$f$ is the function 
\begin{equation}
f(\t_2)=\sqrt{1-8 \cos (2 \t_2)-8 \cos (4 \t_2)}.
\label{f}
\end{equation}
\end{enumerate}
\end{proposition}
\begin{proof}
If $r_1=1$, and at least one of $r_2$ or $r_3$ equals
one, then~\eqref{31} gives $1+e^{2i\t_2}+e^{2i\t_3}=0$
and~\eqref{G} gives $R(r_2)\cos \t_3+R(r_3)\cos \t_2=0$. Since at least one of $r_2$ or $r_3$ equals one, then
at least one of $R_2$ or $R_3$ equals zero. In fact, both $R_2=R_3=0$ (because otherwise we get $\cos \t_2=0$ or 
$\cos \t_3=0$, which prevents $1+e^{2i\t_2}+e^{2i\t_3}$
from cancelling), and thus, $r_2=r_3=1$. In this setting,
$1+e^{2i\t_2}+e^{2i\t_3}=0$ leads to $(c,a_1,a_2,a_3)=(\pm i,1,e^{i\pi/3}, e^{2i\pi/3})$, which proves item 1.

Now assume $\t_2+\t_3=0$. Then~\eqref{31},\eqref{G} give respectively
\begin{eqnarray}
1+2\cos(2\t_2)-R(r_2)R(r_3)&=&R(r_1)[R(r_2)e^{i\t_2}+R(r_3)e^{-i\t_2}],\label{40a}
\\
(R(r_2)+R(r_3))\cos \t_2+R(r_1)\cos(2\t_2)
&=&\frac{1}{2}R(r_1)R(r_2)R(r_3).\label{41}
\end{eqnarray}
Observe that $R(r_1)$ cannot vanish by Lemma~\ref{lem12}
(another reason is that otherwise, \eqref{41}
gives $\cos\t_2=0$, and~\eqref{40a} gives 
$-1-R(r_2)R(r_3)=0$ which is absurd). From~\eqref{40a} we
deduce that $R(r_2)e^{i\t_2}+R(r_3)e^{-i\t_2}$ is real.
This implies that $[R(r_2)-R(r_3)]\sin \t_2=0$. We claim 
that $\sin \t_2\neq 0$; otherwise $\t_2\equiv 0$ (mod $\pi$) and~\eqref{40a},\eqref{41} give the system
\begin{eqnarray}
3-R(r_2)R(r_3)&=&\pm R(r_1)[R(r_2)+R(r_3)],\nonumber
	\\
R(r_1)\pm (R(r_2)+R(r_3))
&=&\frac{1}{2}R(r_1)R(r_2)R(r_3),\nonumber
\end{eqnarray}
(with the same choice for signs), which can be easily seen not to have solutions.

Thus,  $\sin \t_2\neq 0$ hence $R(r_2)=R(r_3)$ and
$r_2=r_3$. In this setting, \eqref{40a},\eqref{41} reduce to
\begin{eqnarray}
1+2\cos(2\t_2)-R(r_2)^2&=&2R(r_1)R(r_2)\cos \t_2,\label{45}
	\\
2R(r_2)\cos \t_2+R(r_1)\cos(2\t_2)
	&=&\frac{1}{2}R(r_1)R(r_2)^2.\label{46}
\end{eqnarray}

If we assume $\t_2+\t_3=\pi$, then~\eqref{31},\eqref{G} give respectively
\begin{eqnarray}
1+2\cos(2\t_2)+R(r_2)R(r_3)&=&R(r_1)[R(r_2)e^{i\t_2}-R(r_3)e^{-i\t_2}],\label{43bis}
\\
(-R(r_2)+R(r_3))\cos \t_2-R(r_1)\cos(2\t_2)
&=&\frac{1}{2}R(r_1)R(r_2)R(r_3).\label{44bis}
\end{eqnarray}
Again, $R(r_1)$ can not vanish due to Lemma~\ref{lem12}. From \eqref{43bis}  
we deduce that $R(r_2)e^{i\t_2}-R(r_3)e^{-i\t_2}$ is real.
This implies that $[R(r_2)+R(r_3)]\sin \t_2=0$. We claim 
that $\sin \t_2\neq 0$; otherwise $\t_2\equiv 0$ (mod $\pi$) and~\eqref{43bis},\eqref{44bis} 
give the system
\begin{eqnarray}
3+R(r_2)R(r_3)&=&\pm R(r_1)[R(r_2)-R(r_3)],\nonumber
	\\
-R(r_1)\pm (-R(r_2)+R(r_3))
&=&\frac{1}{2}R(r_1)R(r_2)R(r_3),\nonumber
\end{eqnarray}
(with the same choice for signs), which again has no solutions.
Thus,  $\sin \t_2\neq 0$ hence
$R(r_2)=-R(r_3)$ and $r_2=1/r_3$. In this setting,
\eqref{40a},\eqref{41} reduce again to \eqref{45} and \eqref{46}.

The system \eqref{45},\eqref{46} has two equations and 
three unknowns $r_1,r_2,\t_2$. Next we describe its solutions. Consider the function $f$ given by~\eqref{f}.
Then, 
\[
f(\pi-\t_2)=f(\t_2), \mbox{ for each }\t_2,\quad 
f(\t_{2,0})=0=f(\pi-\t_{2,0}),
\]
 where $\t_{2,0}=\frac{1}{2} \cot ^{-1}\left(\frac{9}{\sqrt{32 \sqrt{10}+95}}\right)\sim 0.499841$, and
the domain of $f$ is $[\t_{2,0},\pi-\t_{2,0}]+\pi \Z$.
The set $\{ \t_2\in [\t_{2,0},\pi-\t_{2,0}] \ | \ f(\t_2)\geq 3\}$ equals
$A:=[\frac{\pi }{4},\frac{\pi }{3}]\cup 
[\frac{2 \pi }{3},\frac{3 \pi }{4}]$. 

The unique solution $(r_1,r_2)$ to the system \eqref{45},\eqref{46} as a 
function of $\t_2$ is given by \eqref{40b}, \eqref{41b} and the opposite
expressions for both $R(r_1(\t_2)),
R(r_2(\t_2))$, which exchange $(r_1,r_2)$ by $(\frac{1}{r_1},\frac{1}{r_2})$.
\end{proof} 

\subsubsection{The one-parameter family of examples in item~2 of 
Proposition~\ref{lema14}}
\label{sec6.2.1}
Observe that the map $\t_2\in (\frac{\pi}{4},\frac{\pi}{3}]\mapsto \pi -\t_2\in 
[\frac{2 \pi }{3},\frac{3 \pi }{4})$ is a diffeomorphism.
Using the notation in item 2 of Proposition~\ref{lema14}, 
for each $\t_2\in (\frac{\pi }{4},\frac{\pi }{3}]$, we have
\begin{equation}
R(r_1(\pi -\t_2))=-R(r_1(\t_2)),\qquad 
R(r_2(\pi -\t_2)))=R(r_2(\t_2)).
\label{2f}
\end{equation}
Each of these lists with $\t_2\in (\frac{\pi }{4},\frac{\pi }{3}]
\cup [\frac{2\pi }{3},\frac{3\pi }{4})$
solves the 1-sided period problem, hence it defines a 
non-orientable, branched minimal surface $H(\t_2)$.
Furthermore, \eqref{2f} implies that 
\begin{equation}
r_1(\pi-\t_2)=\frac{1}{r_1(\t_2)},\qquad 
r_2(\pi-\t_2)=r_2(\t_2).
\label{2fa}
\end{equation}
We claim the surfaces $H(\t_2)$ and $H(\pi-\t_2)$ are congruent. To see this, note that
the set of points $\{a_j,-1/\overline{a_j}\ : \ j=1,2,3\}$ that defines $f$ through \eqref{2.13} and 
generates the surface $H(\t_2)$, is:
\begin{equation}\label{n56}
    \left\lbrace r_1,\frac{-1}{r_1},r_2 e^{i\theta_2},\frac{1}{r_2}e^{i(\pi+\theta_2)},r_2 e^{-i\theta_2},\frac{1}{r_2}e^{i(\pi-\theta_2)}\right\rbrace.
\end{equation}
The analogous set of points for the surface $H(\pi-\t_2)$ is given through \eqref{2fa}:
\[
 \left\lbrace \frac{1}{r_1}, -r_1 ,-r_2 e^{-i\theta_2},\frac{1}{r_2}e^{-i\theta_2},-r_2 e^{-i\theta_2},\frac{1}{r_2}e^{i\theta_2}\right\rbrace,
\]
which is up to sign the set described in \eqref{n56}. Therefore, the function $f$ defined by equation~\eqref{2.13}
and the corresponding function $\wt{f}$ defined by the same formula for the surface $H(\pi-\t_2)$ are related by 
$\wt{f}(-z)=-f(z)$, for each $z\in \C$. Using that $\omega=f\, dz$ and $\wt{\omega}=\wt{f}\, dz$ define, via the 
Weierstrass representation~\eqref{1}, related branched minimal immersions $X=(x_1,x_2,x_3)$ for $H(\t_2)$ and $\wt{X}=
(\wt{x}_1,\wt{x}_2,\wt{x}_3)$ for $H(\pi-\t_2)$, we get that $H(\t_2)$ and $H(\pi-\t_2)$ are congruent.

In the sequel, we will reduce our study to the family $\{ H(\t_2)\ | \ \t_2\in (\frac{\pi }{4},\frac{\pi }{3}]\}$.
From~\eqref{40b}, \eqref{41b}
we have
\[
\lim_{\t_2\to \pi/3^-}R(r_1(\t_2))=
\lim_{\t_2\to \pi/3^-}R(r_2(\t_2))=0,
\]
which implies that 
\[
\lim_{\t_2\to \pi/3^-}H(\t_2)=H_2.
\]

We next identify the limit (after rescaling)
of the surfaces $H(\t_2)$ as $\t_2\to \pi/4^+$. 
We first observe that
\begin{equation}
\lim_{\t_2\to \pi/4^+}R(r_1(\t_2))=-\infty ,\qquad 
\lim_{\t_2\to \pi/4^+}R(r_2(\t_2))=0.
\label{52}
\end{equation}
This implies that the branch point $a_1=a_1(\t_2)$ is tending to zero, hence the limit of $H(\t_2)$ when $\t_2\to \pi/4^+$ (if it exists) 
cannot be an example with complexity $m=2$. Intuitively, it is clear than
the complexity cannot increase when taking limits (even with different 
scales), hence by Theorem~\ref{thm15} it is natural to think that the 
limit of suitable re-scalings of $H(\t_2)$ when $\t_2\to \pi/4^+$ be $H_1$. We next formalize this idea.

Another consequence of~\eqref{52} is that the list $(c,a_1,a_2,a_3)=
(i,r_1(\t_2),r_2(\t_2)e^{i\t_2},r_2(\t_2)e^{-i\t_2})$
converges as $\t_2\to \pi/4^+$ to $(c,a_1,a_2,a_3)=
(i,0,e^{i\pi/4},e^{-i\pi/4})$. 
After applying to $H(\t_2)$ a homothety of ratio
$r_1(\t_2)>0$ (which shrinks to zero), the Weierstrass data of the shrunk surface
$r_1(\t_2)H(\t_2)$ is $(g(z)=z,\ r_1(\t_2)f(z))$,
where $f(z)$ is given by~\eqref{2.13}. For $z\in
\C\setminus \{0\}$ fixed,
\[
\begin{array}{l}
{\displaystyle \lim_{\t_2\to \pi/4^+}r_1(\t_2)f(z)}
\stackrel{\eqref{2.13}}{=}
{\displaystyle \lim_{\t_2\to \pi/4^+}r_1(\t_2)\frac{i}{z^{5}}\prod_{j=1}^{3}(z-a_j)\left(z+\frac{1}{\overline{a_j}}\right)
}
\\
=\frac{i}{z^5}\left(z-e^{i\pi/4}\right)
\left(z+e^{i\pi/4}\right)\left(z-e^{-i\pi/4}\right)\left(z+e^{-i\pi/4}\right)
{\displaystyle \lim_{\t_2\to \pi/4^+}}
(z-r_1(\t_2))\left(r_1(\t_2)z+1\right)
\\
=\frac{i}{z^4}\left(z-e^{i\pi/4}\right)\left(z+e^{i\pi/4}\right)\left(z-e^{-i\pi/4}\right)\left(z+e^{-i\pi/4}\right) :=
\wh{f}(z).
\end{array}
\]
Plugging the Weierstrass data $(g(z)=z\wh{f}\,dz)$ into \eqref{1}, we obtain a parametrization of the limit surface
of $r_1(\t_2)H(\t_2)$ as $\t_2\to \pi/4^+$ in polar coordinates $z=r e^{i\t}$:
\begin{equation}
\wh{X}(r e^{i\theta})=
\left(\begin{array}{c}
	\frac{-\sin\theta}{2}(r-\frac{1}{r})+\frac{ \sin (3\theta)}{6}(r^{3}-\frac{1}{r^{3}}) \vspace{0.2cm}
	\\ 
	-\frac{\cos \theta }{2}(r-\frac{1}{r})-\frac{ \cos (3\theta)}{6}(r^{3}-\frac{1}{r^{3}}) \vspace{0.2cm}
	\\
	-\cos\theta\sin\theta(r^{2}+\frac{1}{r^{2}})
\end{array}\right).
\label{limitm=2}
\end{equation}
We claim that this parametrization generates 
the Henneberg surface $H_1$. To see this, observe that if we first 
perform the change of variables $\theta=\widetilde{\theta}+\pi/4$
and then rotate the surface an angle of $-\frac{\pi}{4}$ around the $x_3$-axis, we get
\begin{align*}
\left(\begin{array}{ccc}
    \cos\left(\frac{\pi}{4}\right) & \sin\left(\frac{\pi}{4}\right) & 0\\
    -\sin\left(\frac{\pi}{4}\right) & \cos\left(\frac{\pi}{4}\right) & 0 \\
    0 & 0 & 1
\end{array}\right)\cdot \wh{X}(r e^{i\left(\widetilde{\theta}+\frac{\pi}{4}\right)})=-\left(\begin{array}{c}
	\frac{\cos\widetilde{\theta}}{2}(r-\frac{1}{r})-\frac{ \cos (3\widetilde{\theta})}{6}(r^{3}-\frac{1}{r^{3}})
	\\
	-\frac{\sin \widetilde{\theta} }{2}(r-\frac{1}{r})-\frac{ \sin (3\widetilde{\theta})}{6}(r^{3}-\frac{1}{r^{3}})
	\\
	\frac{\cos(2\widetilde{\theta})}{2}(r^{2}+\frac{1}{r^{2}})
\end{array}\right),
\end{align*}
which is, up to a sign, the parametrization given in \eqref{H1X} for $H_1$.

\begin{figure}[h]
\centering
\includegraphics[width=5.2cm]{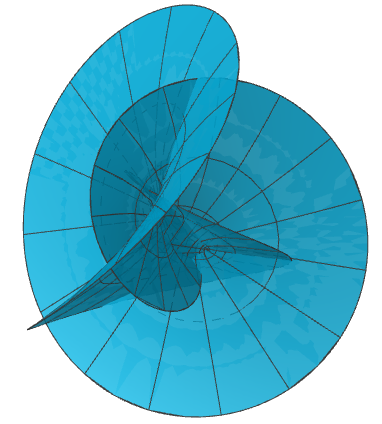}
\includegraphics[width=5.2cm]{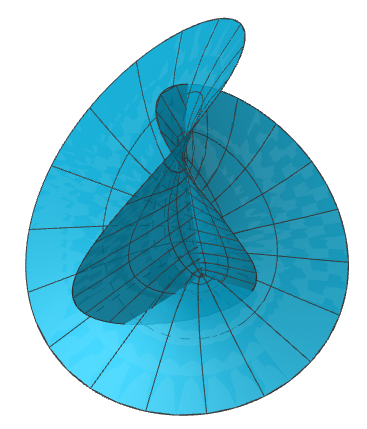}
\includegraphics[width=5.2cm]{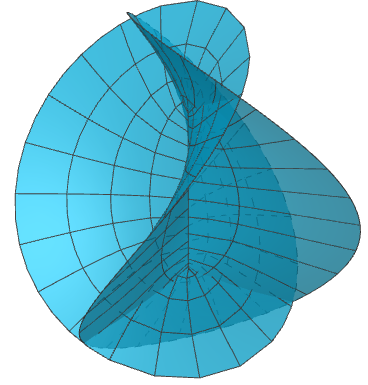}
\caption{Surfaces generated by the previous lists $(c,a_1,a_2,a_3)=
(i,r_1(\t_2),r_2(\t_2)e^{i\t_2},r_2(\t_2)e^{-i\t_2})$ with $\t_2=1$ (left), 
$\t_2=0.83$ (center), $\t_2=0.7854$ (right). 
The limit of $r_1(\t_2)H(\t_2)$ as $\t_2\to \pi/4^+\sim 0.785398$ is the Henneberg surface $H_1$.
}
\label{fig:planar}
\end{figure}

\subsubsection{Around $H_2$ the space of examples with complexity $m=2$ is two-dimensional}
\label{sec6.2.2}

Item~2 of Proposition~\ref{lema14}
defines a non-compact family of non-orientable, branched
minimal surfaces $\{ H(\t_2)\ | \ \t_2\in (\frac{\pi }{4},
\frac{\pi }{3}]\} $ inside the moduli space of examples with complexity $m=2$.
Apparently, the space of solutions for this complexity
has real dimension 2 (the variables are $r_1,r_2,r_3,\t_2,\t_3$, $F=0$ is a 
complex condition  and $G=0$ is a real condition). We can ensure this at 
least around $H_2$ via the implicit function theorem (this is consistent
with item 2 of Proposition~\ref{lema14}, since it imposes the extra 
condition $\t_2+\t_3=0$ mod $\pi$), as we will show next.

Consider the (smooth) period map given by
\[
\begin{array}{rccl}
	P:&(\R^+)^3\times \R^2&\longrightarrow &\R^3\equiv \C\times \R
	\\
	&((r_1,r_2),(r_3,\t_2,\t_3))&\longmapsto & (F(r_1,r_2,r_3,\t_2,\t_3),G(r_1,r_2,r_3,\t_2,\t_3)),
\end{array}
\]
where $F,G$ are given by~\eqref{31}, \eqref{G} respectively. Given 
$(r_1,r_2)\in (\R^+)^2$, let $P^{r_1,r_2}\colon \R^+\times \R^2
\to \R^3$ be the
restriction of $P$ to $\{ (r_1,r_2)\}\times \R^+\times \R^2$. Then,
\begin{equation}
	d(P^{r_1, r_2})_{(r_3,\t_2,\t_3)}\equiv\left(\begin{array}{ccc}
		\displaystyle\frac{\partial \text{Re}(F)}{\partial r_3}&\displaystyle\frac{\partial \text{Re}(F)}{\partial \t_2} &\displaystyle\frac{\partial \text{Re}(F)}{\partial \t_3}\\ \\
		\displaystyle\frac{\partial \text{Im}(F)}{\partial r_3}&\displaystyle\frac{\partial \text{Im}(F)}{\partial \t_2} &\displaystyle\frac{\partial \text{Im}(F)}{\partial \t_3}\\ \\
		\displaystyle\frac{\partial G}{\partial r_3}&\displaystyle\frac{\partial G}{\partial \t_2} &\displaystyle\frac{\partial G}{\partial \t_3}\\ \\
	\end{array}\right).
	\label{50}
\end{equation}
Recall that the list associated to $H_2$ is 
$(r_1,r_2,r_3,\t_2,\t_3)=(1,1,1,\pi/3,
2\pi/3)$. Imposing this choice of parameters and 
computing the determinant of~\eqref{50}
we get
\[
d(P^{1, 1})_{(1,\pi/3,2\pi/3)}=2\sqrt{3}\neq 0.
\]
Thus, the implicit function theorem gives an open neighborhood
$U\subset (\R^+)^2$ of $(r_1,r_2)=(1,1)$, an open set 
$W\subset (\R^+)^3\times \R^2$ with $(r_1,r_2,r_3,\t_2,\t_3)=(1,1,1,\pi/3,
2\pi/3)\in W$ and a smooth map $\varphi \colon U\to \R^3$ such that all the solutions $(r_1,r_2,\t_1,\t_2,\t_3)$ around $(1,1,1,\pi/3,
2\pi/3)$ of the
equation $P(r_1,r_2,\t_1,\t_2,\t_3)=0$ are of the form 
$(\t_1,\t_2,\t_3)=\varphi (r_1,r_2)$. By Remark~\ref{rem12a}(I), 
the list 
\[
(c=e^{i\be(r_1,r_2)},r_1,r_2e^{i\t_2},r_3e^{i\t_3})
\]
with $\be=\be(r_1,r_2)$ given by~\eqref{28}
solves the 1-sided period problem and so, it defines a $1$-sided
branched minimal surface. This produces a
2-parameter deformation of the surface $H_2$ in the moduli space 
of examples with $m=2$ around $H_2$, which in turn describes 
the whole moduli space around $H_2$.

\begin{remark}
{\rm 
A nice consequence of the classical Leibniz formula for 
the derivative of a product is a recursive law that gives
the coefficients of the polynomial $P(z)$ defined
by~\eqref{2.14} in terms of the coefficients of the 
related polynomial for one complexity less. To obtain
this recursive law, we first adapt the notation to the
complexity:
\begin{equation}\label{32}
	P_{m+1}(z):=\prod_{j=1}^{m+1}(z-a_j)\left(z+\frac{1}{\overline{a_j}}\right)=\sum_{h=0}^{2m+2} A_{m+1,h}z^h.
\end{equation}
\eqref{2.15} can now
be written
\begin{equation}\label{33}
	\overline{c A_{m+1,m}}=-c A_{m+1,m+2},\quad \text{Im}(c A_{m+1,m+1})=0.
\end{equation}
We want to find expressions for the above
coefficients $A_{m+1,m},A_{m+1,m+2},A_{m+1,m+1}$, depending only on coefficients of the type $A_{m,h}$
(i.e., for one complexity less). Writing $a_j=r_je^{i\t_j}$ in polar coordinates, observe that
\[
P_{m+1}(z):=P_{m}(z)Q_{m+1}(z),\qquad \mbox{where }
Q_{m+1}(z)=(z-r_{m+1}e^{i\t_{m+1}})\left(z+\frac{e^{i\t_{m+1}}}{r_{m+1}}
\right).
\]
Hence for $h\in \{ m,m+1,m+2\}$,
\begin{eqnarray}
	A_{m+1,h}&=&\frac{1}{h!}P_{m+1}^{(h)}(0)=
	\frac{1}{h!}(P_{m}Q_{m+1})^{(h)}(0)=
	\frac{1}{h!}\sum_{k=0}^h\binom{h}{k}P_{m}^{(k)}(0)
	Q_{m+1}^{(h-k)}(0),\nonumber
\end{eqnarray}
where in the last equality we have used Leibniz formula.
Since $Q_{m+1}$ is a polynomial of degree two, its
derivatives of order three or more vanish. Hence we can reduce the last sum to terms where the index $k$ satisfies $h-k\leq 2$, i.e., $k\in \{ h-2,h-1,h\}$ and thus,
\begin{eqnarray}
	A_{m+1,h}&=&
	\frac{1}{h!}\textstyle{\left[
		\binom{h}{h-2}P_{m}^{(h-2)}(0)Q_{m+1}''(0)
		+\binom{h}{h-1}P_{m}^{(h-1)}(0)Q_{m+1}'(0)
		+\binom{h}{h}P_{m}^{(h)}(0)Q_{m+1}(0)\right]}\nonumber
	\\
	&=&
	\frac{1}{h!}\textstyle{\left[
		\frac{h!}{(h-2)!2}P_{m}^{(h-2)}(0)\cdot 2
		-h\, P_{m}^{(h-1)}(0)R(r_{m+1})e^{i\t_{m+1}}
		-P_{m}^{(h)}(0)e^{2i\t_{m+1}}\right]}\nonumber
	\\
	&=&
	\textstyle{\left[
		\frac{1}{(h-2)!}P_{m}^{(h-2)}(0)
		-\frac{1}{(h-1)!}P_{m}^{(h-1)}(0)R(r_{m+1})e^{i\t_{m+1}}
		-\frac{1}{h!}P_{m}^{(h)}(0)e^{2i\t_{m+1}}\right]}\nonumber
	\\
	&=&
	A_{m,h-2}-A_{m,h-1}R(r_{m+1})e^{i\t_{m+1}}
	-A_{m,h}e^{2i\t_{m+1}},\label{RL}
\end{eqnarray}
which is the desired recurrence law. \eqref{RL} can be used
to find solutions to~\eqref{33} for complexity $m=3$
besides the most symmetric example $H_3$, but the 
equations are complicated and we will not give them here. 
}
\end{remark}

\bibliographystyle{plain}
\bibliography{bill1}
\end{document}